\documentclass[submission]{FPSAC2019}

\usepackage{lipsum}

\usepackage{tikz}
\usetikzlibrary{calc, shapes, backgrounds,arrows,positioning,decorations.pathmorphing}

\newcommand{\whatsuit}{{\diamondsuit}}
\newtheorem{theorem}{Theorem}[section]
\newtheorem{example}[theorem]{Example}
\newtheorem{problem}[theorem]{Problem}
\newtheorem{question}[theorem]{Question}
\newtheorem{conjecture}[theorem]{Conjecture}
\newtheorem{proposition}[theorem]{Proposition}
\newtheorem{lemma}[theorem]{Lemma}

\title{Computational complexity, Newton polytopes, and Schubert polynomials}
\author{Anshul Adve\thanks{\href{mailto:aadve@g.ucla.edu}{aadve@g.ucla.edu}.}\addressmark{1},  Colleen Robichaux\thanks{\href{mailto:cer2@illinois.edu}{cer2@illinois.edu}. Colleen Robichaux was supported by the National Science Foundation Graduate Research Fellowship Program under Grant No. DGE 1746047.
}\addressmark{2}\and Alexander Yong\thanks{\href{ayong@illinois.edu}{ayong@illinois.edu}. Alexander Yong was supported by an NSF grant, a Simons Collaboration grant and a UIUC Campus Research Board grant.
}\addressmark{2}}

\address{\addressmark{1} UCLA, Los Angeles, CA 90095 \\ \addressmark{2} Dept.~of Mathematics, U.~Illinois at Urbana-Champaign, Urbana, IL 61801, USA}

\received{\today}

\abstract{The {\sf nonvanishing} problem asks if a coefficient of a polynomial is nonzero.
Many families of polynomials in algebraic combinatorics admit combinatorial counting rules and simultaneously enjoy having \emph{saturated Newton polytopes} (SNP). Thereby, in amenable cases, {\sf nonvanishing} is in the complexity class ${\sf NP}\cap {\sf coNP}$ of problems with ``good characterizations''. This suggests a new algebraic combinatorics viewpoint on complexity theory.

This paper focuses on the case of \emph{Schubert polynomials}. These form a basis of all polynomials and appear in the study of cohomology rings of flag manifolds. We give a tableau criterion for {\sf nonvanishing}, 
from which we deduce the first polynomial time algorithm. These results are obtained from new characterizations of the \emph{Schubitope}, a generalization of the  permutahedron defined for any
subset of the $n\times n$ grid, together with a theorem of A.~Fink, K.~M\'{e}sz\'aros, and A.~St. Dizier (2018), which proved a conjecture of C.~Monical, N.~Tokcan, and the third author (2017).
}


\usepackage[backend=bibtex]{biblatex}
\begin{document}

\maketitle
\section{Introduction}

The main results of this extended abstract of \cite{ARYadv,ARY} concern Schubert polynomials; these are found in Section~\ref{sec:introschub}. Those results illustrate a general
algebraic combinatorics paradigm for computational complexity theory that we wish to put forward here.

\subsection{Nonvanishing decision problems and SNP}
Algebraic combinatorics studies families of polynomials parameterized by combinatorial objects ${\whatsuit}$
\[
F_{\whatsuit} =\sum_{\alpha\in\mathbb{Z}^n_{\geq0}}c_{\alpha,\whatsuit}x^{\alpha}=\sum_{s\in {\mathcal S}}{\sf wt}(s)\in {\mathbb Z}[x_1,x_2,\ldots,x_n],
\]
each viewed as the multivariate weight generating series for some combinatorially defined set
${\mathcal S}$.

\begin{example}[Schur polynomials]
\label{exa:Schur}
$F_{\whatsuit}=s_{\lambda}$ is a \emph{Schur polynomial}, where $\whatsuit=\lambda$
is an integer partition. 
Here,
${\mathcal S}$ is the set of \emph{semistandard Young tableaux} of shape $\lambda$ with entries in $[n]$, and ${\sf wt}(s)=\prod_i x_i^{\#i\in s}$. Schur polynomials are an important
basis of the vector space of all symmetric polynomials.
\end{example}

\begin{example}[Stanley's chromatic symmetric polynomial]
Another symmetric polynomial is Stanley's \emph{chromatic polynomial} $F_{\whatsuit}=\chi_{G}$ \cite{Sta95}. This time $\whatsuit=G=(V,E)$ is a simple
graph, ${\mathcal S}$ is the set of proper $n$-colorings of $G$, i.e., functions $s:V\to \{1,2,\ldots,n\}$ such that $s(i)\neq s(j)$ if $\{i,j\}\in E$, and ${\sf wt}(s)=\prod_i x_i^{\#s^{-1}(i)}$.
\end{example}

\begin{example}[Schubert polynomials]
The central example  of this paper is non-symmetric. It is the family of \emph{Schubert polynomials} $F_{\whatsuit}={\mathfrak S}_w$, a basis of all polynomials.
Now, $\whatsuit=w$ is a permutation. There are many choices for ${\mathcal S}$, such as the \emph{reduced compatible sequences} of \cite{BJS}. Definitions are given in Section~\ref{sec:introschub}.
\end{example}

\begin{problem}[{\sf nonvanishing}]
\label{prob:nonvanishing}
What is the  complexity of deciding $c_{\alpha,\whatsuit}\neq 0$,
as measured in the input size of $\alpha$ and $\whatsuit$ (under the assumption that arithmetic operations take constant time)?
\end{problem}

In this paper, we give a polynomial time algorithm to determine $c_{\alpha,w}\neq 0$ for the Schubert polynomial.
In general, {\sf nonvanishing} may be undecidable: fix $S \subseteq \mathbb{N}$ that is not recursively enumerable, and let $F_m = \sum_{i=1}^m c_{i,m} x^m$ with $c_{i,m} = 1$ if $i \in S$ and $0$ otherwise. Such sets $S$ exist because there are uncountably many subsets of $\mathbb{N}$, but only countably many algorithms. One can explicitly take $S$ to be the set of halting Turing machines under some numerical encoding \cite{Turing}, or the set of G\"odel encodings \cite{Godel} of statements about $({\mathbb N},+,\times)$ provable in first-order Peano arithmetic. All this said, in our cases of interest, $c_{\alpha,\whatsuit}\in {\mathbb Z}_{\geq 0}$ has \emph{combinatorial positivity}: it is given by a counting rule that implies {\sf nonvanishing} 
is in the class ${\sf NP}$ of problems
with a polynomial time checkable certificate of a {\tt YES} decision. 

Evidently, {\sf nonvanishing} concerns the \emph{Newton polytope},
\[{\sf Newton}(F_{\whatsuit})={\rm conv}\{\alpha:c_{\alpha,\whatsuit}\neq 0\} \subseteq {\mathbb R}^n.\]
C.~Monical, N.~Tokcan and the third author \cite{MTY}  showed that for many examples, $F_{\whatsuit}$ has \emph{saturated Newton polytope} (SNP),  i.e.,
$\gamma\in {\sf Newton}(F_{\whatsuit}) \cap {\mathbb Z}^n \iff c_{\gamma,\whatsuit}\neq 0.$ 
The relevance of SNP to Problem~\ref{prob:nonvanishing}
is:
\begin{quotation}
SNP $\Rightarrow$ {\sf nonvanishing}$(F_{\whatsuit})$ is equivalent to  checking membership of a lattice point in 
${\sf Newton}(F_{\whatsuit})$.
\end{quotation}

\begin{example}[{\sf nonvanishing}$(s_{\lambda})$ is in ${\sf P}$]
\label{exa:inP}
${\sf Newton}(s_{\lambda})$ is the \emph{$\lambda$-permutahedron} ${\mathcal P}_{\lambda}$, the convex hull of the $S_n$-orbit of $\lambda\in {\mathbb R}^n$. 
By symmetry one may assume $\alpha$ is a partition. Thus $c_{\alpha,\lambda}$ is the \emph{Kostka coefficient}, and $c_{\alpha,\lambda}=0$ if and only if $\alpha\leq \lambda$ in \emph{dominance order}. 
So {\sf nonvanishing}$(s_{\lambda})$ is in the class {\sf P} of polynomial time problems.
\end{example}

Does the ``niceness'' of combinatorial positivity and SNP transfer to complexity?

\begin{question}
\label{question:meta}
Under what conditions does combinatorial positivity and SNP of $\{F_{\whatsuit}\}$
imply ${\sf nonvanishing}(F_{\whatsuit})\in{\sf P}$, or at least that ${\sf nonvanishing}(F_{\whatsuit})\not\in {\sf NP}$-{\sf complete}?
\end{question}
 
On the other hand, $\chi_G$ is not generally SNP \cite{MTY} and ${\sf nonvanishing}(\chi_G)$ is hard:

\begin{example}[$\chi_G$-nonvanishing is ${\sf NP}$-{\sf complete}]
\label{exa:inNPcomplete}
For $\chi_G$, nonvanishing is clearly in ${\sf NP}$. In fact, for each fixed $n\geq 3$ it is 
${\sf NP}$-${\sf complete}$. The $n$-coloring problem of deciding if a graph has an $n$-proper coloring is ${\sf NP}$-${\sf complete}$ for each fixed $n\geq 3$. Given an efficient oracle to solve ${\sf nonvanishing}(\chi_G)$, call it 
on each partition of $|V|$ with $n$ parts to decide if there exists a proper $n$-coloring. This requires only
$O(|V|^n)$ calls, so it is a polynomial reduction of $n$-coloring to ${\sf nonvanishing}(\chi_G)$.
\end{example}

\subsection{Context from computer science; connection to Stanley's Schur positivity conjecture}
\label{sec:context}
Examples~\ref{exa:inP} and~\ref{exa:inNPcomplete} show that 
{\sf nonvanishing} can achieve the extremes of ${\sf NP}$. What about the non-extremes? 

The class ${\sf NP}$-{\sf intermediate} consists of ${\sf NP}$ problems that are neither in ${\sf P}$
nor ${\sf NP}$-{\sf complete}. \emph{Ladner's theorem} states that if ${\sf P}\neq {\sf NP}$ there exists an (artificial) ${\sf NP}$-{\sf intermediate} problem. 
Many natural problems from algebra, number theory, game theory and combinatorics 
are \emph{suspected} to be ${\sf NP}$-{\sf intermediate}. An example is the \emph{Graph Isomorphism problem}.

The class {\sf coNP} consists of problems whose complements are in {\sf NP}, i.e., those with a polynomial time checkable certificate of a {\tt NO} decision. 

\begin{quotation}
SNP $\Rightarrow$ given a halfspace description of the Newton polytope, an inequality violation checkable in polynomial time
gives a {\sf coNP} certificate.
\end{quotation}

 The above  implication of SNP says that any solution $\{F_{\whatsuit}\}$ to the following problem
gives ${\sf nonvanishing}(F_{\whatsuit})\in {\sf NP}\cap {\sf coNP}$.

\begin{problem}
\label{prob:findconp}
For a combinatorially positive family of SNP polynomials $\{F_{\whatsuit}\}$,
determine half space descriptions of ${\sf Newton}(F_{\whatsuit})$.
\end{problem}

The class ${\sf NP}\cap {\sf coNP}$ is intriguing.
Membership of a  problem in ${\sf NP}\cap {\sf coNP}$ sometimes foreshadows the harder proof that it is in ${\sf P}$.
For example, consider 
\[\text{${\sf primes}$ $=$ ``is a positive integer $n$ prime?''}\] 
Clearly, ${\sf primes}\in {\sf coNP}$.
In 1975, V.~Pratt \cite{Pratt} showed ${\sf primes}\in{\sf NP}$. It was about thirty years before the celebrated discovery of the \emph{AKS primality test} of M.~Agrawal, N.~Kayal, and 
N.~Saxena \cite{AKS}, establishing ${\sf primes}\in {\sf P}$.

In retrospect, another example is the \emph{linear programming} problem 
\[\text{${\sf LPfeasibility}$ $=$ ``is $A{\bf x}= b, {\bf x}\geq 0$ feasible?''}\]
Clearly ${\sf LPfeasibility}\in{\sf NP}$. Membership in {\sf coNP} is a consequence of \emph{Farkas' Lemma} (1902), which is a foundation for
LP duality, conjectured by J.~von Neumann and proved by G.~Dantzig in 1948 (cf.~\cite{Dantzig}). 
Yet, it was not until 1979, with L.~Khachiyan's work on the \emph{ellipsoid method} that ${\sf LPfeasibility}\in {\sf P}$ was proved; see, e.g., 
the textbook \cite{Schrijver}.

These examples suggest ${\sf P}={\sf NP}\cap {\sf coNP}$. 
However, one has \emph{integer factorization} 
\[\text{${\sf factorization}$ $=$ ``given $1<a<b$ does there exist a divisor $d$ of
$b$ where $1\leq d\leq a$?''}\]
An {\sf NP} certificate is a divisor. A {\sf coNP} certificate is a 
prime factorization (verified using the AKS test). Most 
public key cryptography (such as RSA) relies on ${\sf P}\neq {\sf NP}\cap {\sf coNP}$. 

The debate ${\sf P}\stackrel{?}{=}{\sf NP}\cap {\sf coNP}$ may be rephrased as 
``are problems with good characterizations in {\sf P}?''. One wants new examples of members of ${\sf NP}\cap{\sf coNP}$ that are not known to be in {\sf P}. If such examples are proved to be in ${\sf P}$, this adds evidence for ``$=$''. Yet, relatively few examples are known. In addition to integer
factorization, one has (decision) \emph{Discrete Log}, \emph{Stochastic Games} \cite{Condon}, 
\emph{Parity Games} \cite{Parity} and \emph{Lattice Problems} \cite{AR05}. (It is open whether \emph{Graph Isomorphism} is in ${\sf coNP}$.)
We now connect this discussion with Example~\ref{exa:inNPcomplete}. 

\begin{problem}
\label{prob:colorsub}
Does restricting to a subclass of graphs $G$ where $\chi_G$ is SNP (or Schur positive) change the complexity of $n$-coloring?
\end{problem}


\begin{conjecture}[R.~P.~Stanley \cite{Sta95}]
If $G$ is claw-free (i.e., it contains no induced $K_{1,3}$ subgraph), then $\chi_G$ is Schur positive.
\end{conjecture}

\begin{conjecture}[C.~Monical \cite{Monical:thesis}]
If $\chi_G$ is Schur positive, then it is SNP.
\end{conjecture}

Combining these two conjectures gives
\begin{conjecture}
\label{conj:combined}
If $G$ is claw-free then $\chi_G$ is SNP.
\end{conjecture}

If ${\sf coNP}$ contains an ${\sf NP}$-{\sf complete} problem then 
${\sf NP}={\sf coNP}$ \cite{Goldreich}, solving an open problem with ``$=$''.\footnote{In this circumstance, the (complexity) polynomial hierarchy unexpectedly collapses to the first level.} Now, by \cite{Hol81}, $n$-coloring claw-free graphs is
${\sf NP}$-complete.  Therefore:
\begin{quotation}
If  Conjecture~\ref{conj:combined} holds, Problem~\ref{prob:colorsub} and Question~\ref{question:meta} are answered
negatively. Moreover, a
 solution to Problem~\ref{prob:findconp} 
 proves {\sf nonvanishing}$(\chi_{\text{claw-free\ } G})$ is {\sf coNP}, and hence ${\sf NP}={\sf coNP}$. 
\end{quotation}
This suggests a new complexity-theoretic rationale for the study of  $\chi_G$.

\subsection{An algebraic combinatorics paradigm for complexity}
Summarizing, we are motivated by complexity to study {\sf nonvanishing} in algebraic combinatorics.  Many
polynomial families $\{F_{\whatsuit}\}$ have combinatorial positivity and (conjecturally) SNP \cite{MTY}. 
Together, with a solution to Problem~\ref{prob:findconp}, ${\sf nonvanishing}\in{\sf NP}\cap {\sf coNP}$. 

For each family $\{F_{\whatsuit}\}$, one arrives at one of four logical outcomes, depending on
 the complexity of 
{\sf nonvanishing}$(F_{\whatsuit})$ within ${\sf NP}\cap {\sf coNP}$:

\begin{itemize}
\item[(I)] Unknown: it is a problem, in and of itself, to find additional problems that are in ${\sf NP}\cap {\sf coNP}$ that are not \emph{known} to be in ${\sf P}$.
\item[(II)] {\sf P}: Give an algorithm. It will likely illuminate some special structure, of independent combinatorial interest.
\item[(III)] {\sf NP}-{\sf complete}: proof solves ${\sf NP}\stackrel{?}{=}{\sf coNP}$ with (a suprising) ``$=$''.
\item[(IV)] {\sf NP}-{\sf intermediate}: proof solves {\sf NP}-{\sf intermediate}$\stackrel{?}{=} \emptyset$ with ``$\neq$'' (hence
${\sf P}\neq {\sf NP}$).
\end{itemize}
 
  Our main results in Section~\ref{sec:introschub} illustrate (II) for Schubert polynomials.

\section{Main results: Schubert polynomials}\label{sec:introschub}
\emph{Schubert polynomials} form a  linear basis of all polynomials ${\mathbb Z}[x_1,x_2,x_3,\ldots]$. They were introduced by A.~Lascoux--M.-P.~Sch\"{u}tzenberger \cite{LS1} to study the cohomology ring of the flag manifold. These polynomials
represent the Schubert classes under the Borel isomorphism. A reference is the textbook \cite{Fulton:YT}.

The Schubert polynomial ${\mathfrak S}_{w}(x_1,\ldots,x_n)$ is defined recursively for any permutation $w\in S_n$ as follows.
If $w_0=n \ n-1 \ \cdots 2 \ 1 $ is the longest length permutation 
in $S_n$, then 
\[{\mathfrak S}_{w_0}(x_1,\ldots,x_n):=x_1^{n-1}x_2^{n-2}\cdots x_{n-1}.\] 
Otherwise, $w\neq w_0$ and 
there exists $i$ such that $w(i)<w(i+1)$. Then one sets
\[{\mathfrak S_w}(x_1,\ldots,x_n)=\partial_i {\mathfrak S}_{ws_i}(x_1,\ldots,x_n),\] 
where
$\partial_i f:= \frac{f-s_if}{x_i-x_{i+1}}$,
and $s_i$ is the transposition swapping $i$ and $i+1$. 
Since $\partial_i$ satisfies
\[\partial_i\partial_j=\partial_j\partial_i \text{\ for $|i-j|>1$, and \ } 
\partial_i\partial_{i+1}\partial_i=\partial_{i+1}\partial_i\partial_{i+1},\]
the above description of ${\mathfrak S}_w$ is well-defined. In addition,
under the inclusion 
$\iota: S_n\hookrightarrow S_{n+1}$ defined by
$w(1)\cdots w(n) \mapsto w(1) \ \cdots w(n) \ n+1$,
${\mathfrak S}_w={\mathfrak S}_{\iota(w)}$. 
Thus one unambiguously
refers to ${\mathfrak S}_w$ for each $w\in S_{\infty}=\bigcup_{n\geq 1} S_n$. 

	To each $w\in S_\infty$ there is a unique \emph{code}, 
	${\sf code}(w)=(c_1,c_2,\ldots,c_L)\in {\mathbb Z}_{\geq 0}^L$,
where $c_i$ counts the number of boxes in the $i$-th row of the Rothe diagram $D(w)$ of $w$. If $w$ is the identity then ${\sf code}(w)=\emptyset$;
otherwise,  $c_L>0$ 
(i.e., we truncate any trailing zeroes).

	Now, $c_{\alpha,w}=0$ unless $\alpha_i=0$ for $i>L$, and moreover, $c_{\alpha,w}\in {\mathbb Z}_{\geq 0}$.
Let {\sf Schubert} be the nonvanishing problem for Schubert polynomials. The {\tt INPUT} is ${\sf code}=(c_1,\ldots c_L)\in {\mathbb Z}_{\geq 0}^L$ with $c_L>0$
and $\alpha\in {\mathbb Z}_{\geq 0}^L$. ${\sf Schubert}$ returns {\tt YES} if $c_{\alpha,w}>0$ and {\tt NO} otherwise. 

\begin{theorem}
\label{thm:SchubertinP}
${\sf Schubert}\in {\sf P}$.
\end{theorem}

We prove Theorem~\ref{thm:SchubertinP} using another result. For $w\in S_{n}$, let ${\sf PerfectTab}(D(w),\alpha)$ 
be the fillings of $D(w)$ with $\alpha_k$ many $k$'s, where entries 
in each column are distinct, and any entry in row $i$ is $\leq i$. Let 
${\sf PerfectTab}_{<}(D(w),\alpha) \subseteq {\sf PerfectTab}(D(w),\alpha)$ 
be fillings where entries in each column increase from top to bottom.

\begin{theorem}
\label{thm:Advetableau}
$c_{\alpha,w}>0\iff {\sf PerfectTab}(D(w),\alpha)\neq \emptyset 
\iff{\sf PerfectTab}_{<}(D(w),\alpha)\neq \emptyset$
\end{theorem}

In general $\#{\sf PerfectTab}(D(w),\alpha)\neq c_{\alpha,w}$ but rather 
$\#{\sf PerfectTab}(D(w),\alpha)\geq c_{\alpha,w}$ (cf.~\cite{FGRS}). 

\begin{example}
Here are the tableaux in  $\bigcup_{\alpha} {\sf PerfectTab}_{<}(D(31524),\alpha)$:
\begin{center}
\begin{tikzpicture}[scale=.4]
\draw (0,0) rectangle (5,5);
\draw (0,4) rectangle (1,5) node[pos=.5] {$1$};
\draw (1,4) rectangle (2,5) node[pos=.5] {$1$};
\draw (1,2) rectangle (2,3) node[pos=.5] {$2$};
\draw (3,2) rectangle (4,3) node[pos=.5] {$2$};
\filldraw (2.5,4.5) circle (.5ex);
\draw[line width = .2ex] (2.5,0) -- (2.5,4.5) -- (5,4.5);
\filldraw (0.5,3.5) circle (.5ex);
\draw[line width = .2ex] (0.5,0) -- (0.5,3.5) -- (5,3.5);
\filldraw (4.5,2.5) circle (.5ex);
\draw[line width = .2ex] (4.5,0) -- (4.5,2.5) -- (5,2.5);
\filldraw (1.5,1.5) circle (.5ex);
\draw[line width = .2ex] (1.5,0) -- (1.5,1.5) -- (5,1.5);
\filldraw (3.5,0.5) circle (.5ex);
\draw[line width = .2ex] (3.5,0) -- (3.5,0.5) -- (5,0.5);

\end{tikzpicture}\quad
\begin{tikzpicture}[scale=.4]
\draw (0,0) rectangle (5,5);
\draw (0,4) rectangle (1,5) node[pos=.5] {$1$};
\draw (1,4) rectangle (2,5) node[pos=.5] {$1$};
\draw (1,2) rectangle (2,3) node[pos=.5] {$2$};
\draw (3,2) rectangle (4,3) node[pos=.5] {$1$};
\filldraw (2.5,4.5) circle (.5ex);
\draw[line width = .2ex] (2.5,0) -- (2.5,4.5) -- (5,4.5);
\filldraw (0.5,3.5) circle (.5ex);
\draw[line width = .2ex] (0.5,0) -- (0.5,3.5) -- (5,3.5);
\filldraw (4.5,2.5) circle (.5ex);
\draw[line width = .2ex] (4.5,0) -- (4.5,2.5) -- (5,2.5);
\filldraw (1.5,1.5) circle (.5ex);
\draw[line width = .2ex] (1.5,0) -- (1.5,1.5) -- (5,1.5);
\filldraw (3.5,0.5) circle (.5ex);
\draw[line width = .2ex] (3.5,0) -- (3.5,0.5) -- (5,0.5);
\end{tikzpicture}\quad
\begin{tikzpicture}[scale=.4]
\draw (0,0) rectangle (5,5);

\draw (0,4) rectangle (1,5) node[pos=.5] {$1$};
\draw (1,4) rectangle (2,5) node[pos=.5] {$1$};

\draw (1,2) rectangle (2,3) node[pos=.5] {$3$};

\draw (3,2) rectangle (4,3) node[pos=.5] {$1$};
\filldraw (2.5,4.5) circle (.5ex);
\draw[line width = .2ex] (2.5,0) -- (2.5,4.5) -- (5,4.5);
\filldraw (0.5,3.5) circle (.5ex);
\draw[line width = .2ex] (0.5,0) -- (0.5,3.5) -- (5,3.5);
\filldraw (4.5,2.5) circle (.5ex);
\draw[line width = .2ex] (4.5,0) -- (4.5,2.5) -- (5,2.5);
\filldraw (1.5,1.5) circle (.5ex);
\draw[line width = .2ex] (1.5,0) -- (1.5,1.5) -- (5,1.5);
\filldraw (3.5,0.5) circle (.5ex);
\draw[line width = .2ex] (3.5,0) -- (3.5,0.5) -- (5,0.5);
\end{tikzpicture}
\quad
\begin{tikzpicture}[scale=.4]

\draw (0,0) rectangle (5,5);
\draw (0,4) rectangle (1,5) node[pos=.5] {$1$};
\draw (1,4) rectangle (2,5) node[pos=.5] {$1$};
\draw (1,2) rectangle (2,3) node[pos=.5] {$2$};
\draw (3,2) rectangle (4,3) node[pos=.5] {$3$};
\filldraw (2.5,4.5) circle (.5ex);
\draw[line width = .2ex] (2.5,0) -- (2.5,4.5) -- (5,4.5);
\filldraw (0.5,3.5) circle (.5ex);
\draw[line width = .2ex] (0.5,0) -- (0.5,3.5) -- (5,3.5);
\filldraw (4.5,2.5) circle (.5ex);
\draw[line width = .2ex] (4.5,0) -- (4.5,2.5) -- (5,2.5);
\filldraw (1.5,1.5) circle (.5ex);
\draw[line width = .2ex] (1.5,0) -- (1.5,1.5) -- (5,1.5);
\filldraw (3.5,0.5) circle (.5ex);
\draw[line width = .2ex] (3.5,0) -- (3.5,0.5) -- (5,0.5);
\end{tikzpicture}\quad
\begin{tikzpicture}[scale=.4]

\draw (0,0) rectangle (5,5);

\draw (0,4) rectangle (1,5) node[pos=.5] {$1$};
\draw (1,4) rectangle (2,5) node[pos=.5] {$1$};
\draw (1,2) rectangle (2,3) node[pos=.5] {$3$};
\draw (3,2) rectangle (4,3) node[pos=.5] {$2$};

\filldraw (2.5,4.5) circle (.5ex);
\draw[line width = .2ex] (2.5,0) -- (2.5,4.5) -- (5,4.5);
\filldraw (0.5,3.5) circle (.5ex);
\draw[line width = .2ex] (0.5,0) -- (0.5,3.5) -- (5,3.5);
\filldraw (4.5,2.5) circle (.5ex);
\draw[line width = .2ex] (4.5,0) -- (4.5,2.5) -- (5,2.5);
\filldraw (1.5,1.5) circle (.5ex);
\draw[line width = .2ex] (1.5,0) -- (1.5,1.5) -- (5,1.5);
\filldraw (3.5,0.5) circle (.5ex);
\draw[line width = .2ex] (3.5,0) -- (3.5,0.5) -- (5,0.5);
\end{tikzpicture}\quad
\begin{tikzpicture}[scale=.4]

\draw (0,0) rectangle (5,5);

\draw (0,4) rectangle (1,5) node[pos=.5] {$1$};
\draw (1,4) rectangle (2,5) node[pos=.5] {$1$};

\draw (1,2) rectangle (2,3) node[pos=.5] {$3$};

\draw (3,2) rectangle (4,3) node[pos=.5] {$3$};

\filldraw (2.5,4.5) circle (.5ex);
\draw[line width = .2ex] (2.5,0) -- (2.5,4.5) -- (5,4.5);
\filldraw (0.5,3.5) circle (.5ex);
\draw[line width = .2ex] (0.5,0) -- (0.5,3.5) -- (5,3.5);
\filldraw (4.5,2.5) circle (.5ex);
\draw[line width = .2ex] (4.5,0) -- (4.5,2.5) -- (5,2.5);
\filldraw (1.5,1.5) circle (.5ex);
\draw[line width = .2ex] (1.5,0) -- (1.5,1.5) -- (5,1.5);
\filldraw (3.5,0.5) circle (.5ex);
\draw[line width = .2ex] (3.5,0) -- (3.5,0.5) -- (5,0.5);
\end{tikzpicture}
\end{center}

Hence, for instance, $c_{(2,1,1),31524}>0$ but $c_{(4),31524}=0$.
\end{example}

To prove Theorems~\ref{thm:SchubertinP} and~\ref{thm:Advetableau} we establish more general results about the \emph{Schubitope} introduced in \cite{MTY}. This polytope ${\mathcal S}_D$ generalizes the $\lambda$-permutahedron of Example~\ref{exa:inP}. 
It is defined with a halfspace description for any
diagram of boxes $D\subseteq [n]^2$. 

  In the case of Rothe diagrams $D:=D(w)$, it was conjectured in \cite{MTY} that ${\mathcal S}_{D(w)}$
is the Newton polytope of ${\mathfrak S}_w$ and moreover that ${\mathfrak S}_w$ has the SNP property.
These conjectures were proved by A.~Fink-K.~M\'esz\'aros-A.~St.~Dizier \cite{Fink}.  
This, combined with Theorem \ref{thm:independent_characterization} and properties of perfect tableaux, proves Theorem~\ref{thm:Advetableau}.

\emph{Key polynomials} $\kappa_\beta$ are a specialization of the non-symmetric Macdonald polynomials.
Similarly to the above case, for skyline diagrams  $D:=D_\beta$, \cite{MTY} conjectured that ${\mathcal S}_{D_\beta}$
is the Newton polytope of $\kappa_\beta$ and moreover that $\kappa_\beta$ are SNP; this is proved in \cite{Fink}.  Nonvanishing is also in ${\sf P}$, provable using results of 
Section~\ref{subsctn:polytopes} in a manner analogous to that used for the Schubert polynomials.

\section{The Schubitope}\label{sec:SchubChar}

	Fix $n \in \mathbb{Z}_{>0}$ and let $D \subseteq [n]^2$. We call $D$ a \emph{diagram} and visualize $D$ as a subset of an $n \times n$ grid of boxes, oriented so that $(r,c) \in [n]^2$ represents the box in the $r$th row from the top and the $c$th column from the left. Given $S \subseteq [n]$ and a column $c \in [n]$, construct a string denoted ${\sf word}_{c,S}(D)$ by reading column $c$ from top to bottom and recording
	\begin{itemize}
		\item $($ if $(r,c) \not\in D$ and $r \in S$,
		\item $)$ if $(r,c) \in D$ and $r \not\in S$, and
		\item $\star$ if $(r,c) \in D$ and $r \in S$.
	\end{itemize}

	Let $\theta_D^c(S) = \#\{\star\text{'s in } {\sf word}_{c,S}(D)\} + \#\{\text{paired } ()\text{'s in } {\sf word}_{c,S}(D)\}$ and
	\begin{align*}
	\theta_D(S) = \sum_{c=1}^{n} \theta_D^c(S).
	\end{align*}
	
\begin{example}\label{ex:schubWord} In the diagram $D$ below, we labelled the corresponding strings for ${\sf word}_{c,S}(D)$ for $S=\{1,3\}$. For instance, we see ${\sf word}_{5,\{1,3\}}(D)=(\star)$.
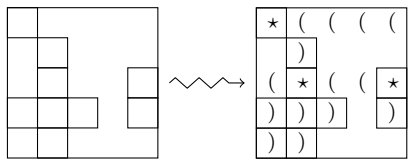
\begin{center}
\begin{tikzpicture}

\node (root) {\begin{tikzpicture}[scale=.4]
\draw (0,0) rectangle (5,5);

\draw (0,4) rectangle (1,5);

\draw (1,3) rectangle (2,4);

\draw (1,2) rectangle (2,3);
\draw (4,2) rectangle (5,3);

\draw (0,1) rectangle (1,2);
\draw (1,1) rectangle (2,2);
\draw (2,1) rectangle (3,2);
\draw (4,1) rectangle (5,2);

\draw (1,0) rectangle (2,1);
\end{tikzpicture}};

\node[right=1 of root] (temp) {\begin{tikzpicture}[scale=.4]
\draw (0,0) rectangle (5,5);

\draw (0,4) rectangle (1,5);

\draw (1,3) rectangle (2,4);

\draw (1,2) rectangle (2,3);
\draw (4,2) rectangle (5,3);

\draw (0,1) rectangle (1,2);
\draw (1,1) rectangle (2,2);
\draw (2,1) rectangle (3,2);
\draw (4,1) rectangle (5,2);

\draw (0,0) rectangle (1,1);
\draw (1,0) rectangle (2,1);

\node at (0,4.5) {\tiny$\star$};
\node at (1,4.5) {\tiny$($};
\node at (2,4.5) {\tiny$($};
\node at (3,4.5) {\tiny$($};
\node at (4,4.5) {\tiny$($};

\node at (0,2.5) {\tiny$($};
\node at (1,2.5) {\tiny$\star$};
\node at (2,2.5) {\tiny$($};
\node at (3,2.5) {\tiny$($};
\node at (4,2.5) {\tiny$\star$};

\node at (1,3.5) {\tiny$)$};

\node at (0,1.5) {\tiny$)$};
\node at (1,1.5) {\tiny$)$};
\node at (2,1.5) {\tiny$)$};
\node at (4,1.5) {\tiny$)$};

\node at (0,0.5) {\tiny$)$};
\node at (1,0.5) {\tiny$)$};

\end{tikzpicture}};

\draw [->,
line join=round,
decorate, decoration={
    zigzag,
    segment length=9,
    amplitude=2,post=lineto,
    post length=2pt
}] (root) -- (temp);
\path (temp);

\end{tikzpicture}
\end{center}
\end{example}

	The \emph{Schubitope} $\mathcal{S}_D$, as defined in \cite{MTY}, is the polytope
	\begin{align}\label{eqn:ogschubitope}
	\left\{(\alpha_1,\dots,\alpha_n) \in \mathbb{R}_{\geq 0}^n : \alpha_1 + \dots + \alpha_n = \#D \text{ and } \sum_{i \in S} \alpha_i \leq \theta_D(S) \text{ for all } S \subseteq [n]\right\}.
	\end{align}
	
	\subsection{Characterizations via tableaux}
	
	A \emph{tableau} of \emph{shape} $D$ is a map 
\[\tau : D \rightarrow [n] \cup \{\circ\},\] 
where $\tau(r,c) = \circ$ indicates that the box $(r,c)$ is unlabelled. Let ${\sf Tab}(D)$ denote the set of such tableaux.
		One of the ideas in our proofs is to reformulate the original definition of $\theta_D(S)$ into the language of tableaux. Given $S \subseteq [n]$, define $\pi_{D,S} \in {\sf Tab}(D)$ by
	\begin{align}\label{eqn:pi_definition}
	\pi_{D,S}(r,c) :=
	\begin{cases}
	r &\text{if } (r,c) \text{ contributes a ``$\star$''  to } {\sf word}_{c,S}(D), \\
	s &\mbox{if } (r,c) \text{ contributes a  ``)'' to } {\sf word}_{c,S}(D) \mbox{ which is} \\
	&\text{paired with an ``('' from } (s,c), \\
	\circ &\mbox{otherwise}.
	\end{cases}
	\end{align}
	In (\ref{eqn:pi_definition}) we pair by the standard ``inside-out'' convention.
\begin{example} \label{ex:schubTab} Continuing Example \ref{ex:schubWord}, below is $\pi_{D,\{1,3\}}(D)$.
\begin{center}
    \begin{tikzpicture}[scale=.4]
\draw (0,0) rectangle (5,5);

\draw (0,4) rectangle (1,5)  node[pos=.5] {$1$};

\draw (1,3) rectangle (2,4)  node[pos=.5] {$1$};

\draw (1,2) rectangle (2,3)  node[pos=.5] {$3$};
\draw (4,2) rectangle (5,3)  node[pos=.5] {$3$};

\draw (0,1) rectangle (1,2)  node[pos=.5] {$3$};
\draw (1,1) rectangle (2,2)  node[pos=.5] {$\circ$};
\draw (2,1) rectangle (3,2)  node[pos=.5] {$3$};
\draw (4,1) rectangle (5,2)  node[pos=.5] {$1$};

\draw (1,0) rectangle (2,1)  node[pos=.5] {$\circ$};

\draw (0,0) rectangle (1,1)  node[pos=.5] {$\circ$};

\end{tikzpicture}
\end{center}
\end{example}

	\begin{theorem}\label{thm:dependent_characterization}
		Let $D \subseteq [n]^2$ and $\alpha = (\alpha_1,\dots,\alpha_n) \in \mathbb{Z}_{\geq 0}^n$ with $\alpha_1 + \dots + \alpha_n = \#D$. Then $\alpha \in \mathcal{S}_D$ if and only if for each $S \subseteq [n]$,  $\sum_{i \in S} \alpha_i \leq \#\pi_{D,S}^{-1}(S)$.
	\end{theorem}

	Define $\tau \in {\sf Tab}(D)$ to be \emph{flagged} if $\tau(r,c) \leq r$ whenever $\tau(r,c) \neq \circ$. It is \emph{column-injective} if $\tau(r,c) \neq \tau(r',c)$ whenever $r \neq r'$ and $\tau(r,c) \neq \circ$. 

\begin{example}\label{ex:FCIT}
Of the tableaux of shape $D$ below, only the second and fourth are flagged, and only the third and fourth are column-injective.
\begin{center}
\begin{tikzpicture}[scale=.4]

\draw (0,0) rectangle (5,5);
\draw (0,4) rectangle (1,5) node[pos=.5] {$1$};
\draw (1,4) rectangle (2,5) node[pos=.5] {$1$};
\draw (4,3) rectangle (5,4) node[pos=.5] {$2$};
\draw (1,2) rectangle (2,3) node[pos=.5] {$5$};
\draw (3,2) rectangle (4,3) node[pos=.5] {$4$};
\draw (4,2) rectangle (5,3) node[pos=.5] {$\circ$};

\draw (1,1) rectangle (2,2) node[pos=.5] {$2$};

\draw (3,0) rectangle (4,1) node[pos=.5] {$4$};
\end{tikzpicture}\qquad
\begin{tikzpicture}[scale=.4]

\draw (0,0) rectangle (5,5);
\draw (0,4) rectangle (1,5) node[pos=.5] {$1$};
\draw (1,4) rectangle (2,5) node[pos=.5] {$1$};
\draw (4,3) rectangle (5,4) node[pos=.5] {$2$};
\draw (1,2) rectangle (2,3) node[pos=.5] {$3$};
\draw (3,2) rectangle (4,3) node[pos=.5] {$2$};
\draw (4,2) rectangle (5,3) node[pos=.5] {$\circ$};
\draw (1,1) rectangle (2,2) node[pos=.5] {$2$};

\draw (3,0) rectangle (4,1) node[pos=.5] {$2$};

\end{tikzpicture}\qquad
\begin{tikzpicture}[scale=.4]

\draw (0,0) rectangle (5,5);

\draw (0,4) rectangle (1,5) node[pos=.5] {$1$};
\draw (1,4) rectangle (2,5) node[pos=.5] {$1$};

\draw (4,3) rectangle (5,4) node[pos=.5] {$2$};
\draw (1,2) rectangle (2,3) node[pos=.5] {$5$};
\draw (3,2) rectangle (4,3) node[pos=.5] {$4$};
\draw (4,2) rectangle (5,3) node[pos=.5] {$\circ$};
\draw (1,1) rectangle (2,2) node[pos=.5] {$\circ$};

\draw (3,0) rectangle (4,1) node[pos=.5] {$3$};

\end{tikzpicture}
\qquad
\begin{tikzpicture}[scale=.4]

\draw (0,0) rectangle (5,5);

\draw (0,4) rectangle (1,5) node[pos=.5] {$1$};
\draw (1,4) rectangle (2,5) node[pos=.5] {$1$};

\draw (4,3) rectangle (5,4) node[pos=.5] {$\circ$};

\draw (1,2) rectangle (2,3) node[pos=.5] {$3$};
\draw (3,2) rectangle (4,3) node[pos=.5] {$3$};
\draw (4,2) rectangle (5,3) node[pos=.5] {$\circ$};

\draw (1,1) rectangle (2,2) node[pos=.5] {$2$};

\draw (3,0) rectangle (4,1) node[pos=.5] {$4$};
\end{tikzpicture}
\end{center}
\end{example}

	Further, call a tableau $\tau \in {\sf Tab}(D)$ \emph{perfect} if $\tau$ is flagged, column-injective, and if no boxes are left unlabelled, i.e., $\tau^{-1}(\{\circ\}) = \emptyset$. Say $\tau \in {\sf Tab}(D)$ has \emph{content} $\alpha$ if $\#\tau^{-1}(\{i\}) = \alpha_i$ for each $i \in [n]$. Let ${\sf PerfectTab}(D,\alpha)$ denote the set of perfect tableaux of content $\alpha$.

We use Theorem~\ref{thm:dependent_characterization} to prove:

	\begin{theorem}\label{thm:independent_characterization}
		Let $D \subseteq [n]^2$ and $\alpha = (\alpha_1,\dots,\alpha_n) \in \mathbb{Z}_{\geq 0}^n$. Then $\alpha \in \mathcal{S}_D$ if and only if ${\sf PerfectTab}(D,\alpha) \neq \emptyset$.
	\end{theorem}

	\section{Polytopal descriptions of perfect tableaux}\label{subsctn:polytopes}
	By Theorem \ref{thm:independent_characterization}, to decide $\alpha \in \mathcal{S}_D$, it suffices to determine ${\sf PerfectTab}(D,\alpha) \neq \emptyset$. Thus it remains to analyze the complexity of deciding ${\sf PerfectTab}(D,\alpha)\neq \emptyset$. 
	
	For this, we construct a polytope that characterizes ${\sf PerfectTab}(D,\alpha)$.
	Given $D \subseteq [n]^2$ and $\alpha = (\alpha_1,\dots,\alpha_n) \in \mathbb{Z}_{\geq 0}^n$, define 
	\[\mathcal{P}(D,\alpha) \subseteq \mathbb{R}^{n^2}\] 
	to be the polytope with points of the form
	$
	(\alpha_{ij})_{i,j\in[n]} = (\alpha_{11},\dots,\alpha_{n1},\dots,\alpha_{1n},\dots,\alpha_{nn})
	$
	governed by the inequalities (A)-(C) below.
	\begin{enumerate}
		\item[(A)] Column-Injectivity Conditions: For all $i,j \in [n]$,
		\begin{align*}
		0 \leq \alpha_{ij} \leq 1.
		\end{align*}
		
		\item[(B)] Content Conditions: For all $i \in [n]$,
		\begin{align*}
		\sum_{j=1}^n \alpha_{ij} = \alpha_i.
		\end{align*}
		
		\item[(C)] Flag Conditions: For all $s,j \in [n]$,
		\begin{align*}
		\sum_{i=1}^s \alpha_{ij} \geq \#\{(i,j) \in D : i \leq s\}.
		\end{align*}
	\end{enumerate}

	\begin{theorem}\label{thm:int_pt_characterization}
		Let $D \subseteq [n]^2$ and $\alpha = (\alpha_1,\dots,\alpha_n) \in \mathbb{Z}_{\geq 0}^n$. Then ${\sf PerfectTab}(D,\alpha) \neq \emptyset$ if and only if $\alpha_1 + \dots + \alpha_n = \#D$ and $\mathcal{P}(D,\alpha) \cap \mathbb{Z}^{n^2} \neq \emptyset$.
	\end{theorem}

	Theorem \ref{thm:int_pt_characterization} formulates the problem of determining if
	${\sf PerfectTab}(D,\alpha)\neq \emptyset$ in terms of feasibility of an integer linear programming problem. 
	In general, integral feasibility is ${\sf NP}$-{\sf complete}. However, ${\mathcal P}(D,\alpha)$ is totally unimodular. Thus
	feasibility of the problem is equivalent to feasibility of its LP-relaxation:
	
	\begin{theorem}\label{thm:relaxation_equivalence}
		Let $D \subseteq [n]^2$ and $\alpha = (\alpha_1,\dots,\alpha_n) \in \mathbb{Z}^n$ with $\alpha_1 + \dots + \alpha_n = \#D$. Then $\mathcal{P}(D,\alpha) \cap \mathbb{Z}^{n^2} \neq \emptyset$ if and only if $\mathcal{P}(D,\alpha) \neq \emptyset$.
	\end{theorem}

	By combining Theorems \ref{thm:independent_characterization}, \ref{thm:int_pt_characterization}, and \ref{thm:relaxation_equivalence} with the fact that $\mathcal{P}(D,\alpha)$ has a polynomial size halfspace description, it follows that $\alpha \in \mathcal{S}_D$ can be decided in $n^{O(1)}$-time. However, this result can be improved. 
	If $D\subseteq [n]^2$ has many identical columns, then many of the flag conditions (C) will look essentially the same. Therefore, our final goal will be to construct a ``compressed" version of $\mathcal{P}(D,\alpha)$ that removes some of the repetitive inequalities.

								A tuple ${\mathcal C}=(m,\{P_k\}_{k=1}^\ell)$ is a \emph{compression} of $D \subseteq [n]^2$ if:
	\begin{itemize}
		\item $m \leq n$ is a nonnegative integer such that $(r,p) \not\in D$ for $r > m$ and $p \in [n]$, and
		
		\item $P=P_1\dot{\cup}\cdots\dot{\cup}P_{\ell} \subseteq [n]$ such that if  $p,p' \in P_k$ then
	\[\{r \in [n] : (r,p) \in D\} = \{r \in [n] : (r,p') \in D\},\]
and moreover if $D$ is nonempty in column $p$ then $p\in P_k$ for some $k\in [\ell]$.
	\end{itemize}

	For a compression ${\mathcal C}$ of $D\subseteq [n]^2$ and 
$\tilde{\alpha} = (\tilde{\alpha}_1,\dots,\tilde{\alpha}_m) \in \mathbb{Z}_{\geq 0}^m$ define
	\begin{align}\label{eqn:compressed_arguments}
		\mathcal{Q}(D,{\mathcal C},\tilde{\alpha}) \subseteq \mathbb{R}^{m\ell}
	\end{align}
	to be the polytope with points of the form $(\tilde{\alpha}_{ik})_{i\in[m],k\in[\ell]}$ satisfying (A')-(C') below.
	
	\begin{enumerate}
		\item[(A')] Column-Injectivity Conditions: For all $i \in [m], k \in [\ell]$,
		\begin{align*}
		0 \leq \tilde{\alpha}_{ik} \leq 1.
		\end{align*}
		
		\item[(B')] Content Conditions: For all $i \in [m]$,
		\begin{align*}
		\sum_{k=1}^{\ell} \#P_k\cdot  \tilde{\alpha}_{ik} = \alpha_i.
		\end{align*}
		
		\item[(C')] Flag Conditions: For all $s \in [m], k \in [\ell]$,
		\begin{align*}
		\sum_{i=1}^s \tilde{\alpha}_{ik} \geq \#\{(i,p_k) \in D : i \leq s, p_k:=\min P_k \}.
		\end{align*}
	\end{enumerate}

	\begin{theorem}\label{thm:compression_equivalence} Let $\alpha = (\alpha_1,\dots,\alpha_n) \in \mathbb{Z}_{\geq 0}^n$ and $\tilde{\alpha} = (\tilde{\alpha}_1,\dots,\tilde{\alpha}_m) := (\alpha_1,\dots,\alpha_m)$. Then
		$\alpha_1 + \dots + \alpha_n = \#D$ and $\mathcal{P}(D,\alpha) \neq \emptyset$ if and only if $\alpha_1 + \dots + \alpha_m = \#D$, $\alpha_{m+1} = \dots = \alpha_n = 0$, and $\mathcal{Q}(D,{\mathcal C},\tilde{\alpha}) \neq \emptyset$.
	\end{theorem}

	\subsection{Deciding membership in the Schubitope}\label{subsec:polyDec}
	
	We use the above results to give a polynomial time algorithm to check if a lattice point is in the Schubitope. This more general result gives a polynomial time algorithm for any polynomial family whose Newton polytopes are Schubitopes.
			Let $D \subseteq [n]^2$, and fix a compression ${\mathcal C}=(m,\{P_k\}_{k=1}^{\ell})$ of $D$.

	\begin{theorem}\label{thm:lp_characterization}
		Let $\alpha = (\alpha_1,\dots,\alpha_n) \in \mathbb{Z}_{\geq 0}^n$. Then $\alpha \in \mathcal{S}_D$ if and only if $\alpha_1 + \dots + \alpha_m = \#D$, $\alpha_{m+1} = \dots = \alpha_n = 0$, and $\mathcal{Q}(D,{\mathcal C},\tilde{\alpha}) \neq \emptyset$, where $\tilde{\alpha} = (\tilde{\alpha}_1,\dots,\tilde{\alpha}_m) := (\alpha_1,\dots,\alpha_m)$.
	\end{theorem}

	For each $k \in [\ell]$, let $p_k:=\min P_k$ and $R_k({\mathcal C}):= \{r \in [n] : (r,p_k) \in D\} \subseteq [m]$.

	\begin{theorem}\label{thm:poly_time_schubitope}
		Given as input $\{R_k({\mathcal C})\}_{k=1}^{\ell}$, $\{\# P_k\}_{k=1}^{\ell}$, and $\tilde{\alpha} = (\tilde{\alpha}_1,\dots,\tilde{\alpha}_m) \in \mathbb{Z}_{\geq 0}^m$ satisfying $\tilde{\alpha}_1 + \dots + \tilde{\alpha}_m = \#D$, one can decide if $\alpha := (\tilde{\alpha}_1,\dots,\tilde{\alpha}_m,0,\dots,0) \in \mathbb{Z}_{\geq 0}^n$ lies in $\mathcal{S}_D$ in polynomial time in $m$ and $\ell$.
	\end{theorem}

	\section{Application to $D(w)$: proof of Theorems~\ref{thm:SchubertinP} and~\ref{thm:Advetableau}}
 For the specialization to Rothe diagrams $D:=D(w)$, the results of A.~Fink-K.~M\'esz\'aros-A.~St.~Dizier \cite{Fink} imply
  \[\alpha \in \mathcal{S}_{D(w)} \iff c_{\alpha,w}>0.\]
Combining this with Theorem \ref{thm:independent_characterization}, 
\[c_{\alpha,w}>0 \iff {\sf PerfectTab}(D(w),\alpha) \neq \emptyset.\]
Further, if ${\sf PerfectTab}(D(w),\alpha) \neq \emptyset$, we can find $\tau \in {\sf PerfectTab}(D(w),\alpha)$ which is also strictly increasing along columns. Thus ${\sf PerfectTab}_{>}(D(w),\alpha)\neq \emptyset$, and Theorem \ref{thm:Advetableau} follows.

To obtain Theorem \ref{thm:SchubertinP} we apply the results of Section \ref{subsctn:polytopes} to $D(w)$. Suppose ${\sf code}(w)=(c_1,\ldots,c_L)$. Let $\sigma\in S_L$ be such that 
\[w(\sigma(1))<w(\sigma(2))<\ldots< w(\sigma(L)).\] Set $w(\sigma(0)):=0$. 
The key lemma we need is:

\begin{lemma}\label{prop:colTypeRothe}
For $1\leq h\leq L$, and for all 
\[j_1,j_2\in\{w(\sigma(h-1))+1,w(\sigma(h-1))+2,\ldots,w(\sigma(h))-1\},\]
we have $(i,j_1)\in D(w)$ if and only if $(i,j_2)\in D(w)$.
\end{lemma}

Using Lemma \ref{prop:colTypeRothe}, there exists a compression ${\mathcal C}=(L,\{P_k\}_{k=1}^{\ell})$ of $D(w)$
where $\ell\leq 2L$. With the following statement, Theorem \ref{thm:poly_time_schubitope} proves Theorem \ref{thm:SchubertinP}.

\begin{proposition}\label{prop:comprRothePoly}
There exists an $O(L^{2})$-time algorithm to compute
${\mathcal C}, \{\# P_k\}_{k=1}^{\ell}$, and $\{R_k\}_{k=1}^{\ell}$ from the input ${\sf code}(w)=(c_1,\ldots,c_L)$. 
\end{proposition}

\acknowledgements{We acknowledge Mathoverflow, 
S.~Kintali's blog ``My brain is open,'' and R.~O'Donnell's Youtube videos (from his class at Carnegie Mellon)
for background. We thank Philipp Hieronymi, Alexandr Kostochka, Cara Monical, Erik Walsberg, Douglas West, Alex-ander Woo, and the anonymous referees for their helpful comments.}

\printbibliography
\end{document}